\documentclass[a4paper,reqno,10pt]{amsart}

\usepackage{a4wide}
\usepackage{hyperref}

\usepackage{amssymb}
\usepackage{amstext}
\usepackage{amsmath}
\usepackage{amscd}
\usepackage{amsthm}
\usepackage{amsfonts}

\usepackage{enumitem}
\setenumerate[1]{label={\upshape(\arabic*)}}
\setenumerate[2]{label={\upshape(\alph*)}}

\usepackage{tikz}
\usetikzlibrary{cd,arrows,decorations.pathmorphing,decorations.pathreplacing}

\newtheorem{theorem}{Theorem}[section]
\newtheorem{theoremi}{Theorem}
\newtheorem{corollaryi}[theoremi]{Corollary}
\newtheorem{propositioni}[theoremi]{Proposition}

\newtheorem{corollary}[theorem]{Corollary}
\newtheorem{lemma}[theorem]{Lemma}
\newtheorem{proposition}[theorem]{Proposition}
\newtheorem{definition-proposition}[theorem]{Definition-Proposition}

\theoremstyle{definition}
\newtheorem{definition}[theorem]{Definition}

\newtheorem{remark}[theorem]{Remark}

\newtheorem*{ack}{Acknowledgement}
\newtheorem*{conv}{Conventions}
\newtheorem*{org}{Organization}

\newcommand{\qedb}{\hfill\blacksquare}

\newcommand{\CC}{\mathcal{C}}

\newcommand{\FF}{\mathcal{F}}

\renewcommand{\SS}{\mathcal{S}}

\newcommand{\Ext}{\operatorname{Ext}\nolimits}

\newcommand{\op}{\operatorname{op}\nolimits}
\newcommand{\RHom}{\mathbf{R}\strut\kern-.2em\operatorname{Hom}\nolimits}

\newcommand{\Image}{\operatorname{Im}\nolimits}
\newcommand{\Kernel}{\operatorname{Ker}\nolimits}
\newcommand{\Cokernel}{\operatorname{Coker}\nolimits}

\newcommand{\Ab}{\mathcal{A}b}
\newcommand{\coker}{\Cokernel}
\newcommand{\im}{\Image}
\renewcommand{\ker}{\Kernel}
\newcommand{\un}{\underline}

\DeclareMathOperator{\Mod}{\mathsf{Mod}}

\DeclareMathOperator{\coh}{\mathsf{coh}}

\DeclareMathOperator{\id}{\mathsf{id}}

\DeclareMathOperator{\Eff}{\mathsf{Eff}}
\DeclareMathOperator{\ddef}{\mathsf{def}}

\newcommand{\iso}{\cong}

\newcommand{\defl}{\twoheadrightarrow}

\newenvironment{sbmatrix}{\left[\begin{smallmatrix}}{\end{smallmatrix}\right]}

\newcommand{\E}{\mathbb{E}}
\newcommand{\F}{\mathbb{F}}
\renewcommand{\ss}{\mathfrak{s}}

\newcommand{\EE}{\mathcal{E}}

\numberwithin{equation}{section}

\begin{document}
\title[Classifying substructures of extriangulated categories via Serre subcategories]{Classifying substructures of extriangulated categories via Serre subcategories}

\author[H. Enomoto]{Haruhisa Enomoto}

\address{Graduate School of Mathematics, Nagoya University, Chikusa-ku, Nagoya. 464-8602, Japan}
\email{m16009t@math.nagoya-u.ac.jp}
\begin{abstract}
  We give a classification of substructures (= closed subbifunctors) of a given skeletally small extriangulated category by using the category of defects, in a similar way to the author's classification of exact structures of a given additive category. More precisely, for an extriangulated category, possible substructures are in bijection with Serre subcategories of an abelian category consisting of defects of conflations. As a byproduct, we prove that for a given skeletally small additive category, the poset of exact structures on it is isomorphic to the poset of Serre subcategories of some abelian category.
\end{abstract}

\maketitle

\tableofcontents

\section{Introduction}

Recently, Nakaoka and Palu \cite{NP} introduced an \emph{extriangulated category} as a simultaneous
generalization of exact categories and triangulated categories. An extriangulated category consists of a triple $(\CC,\E,\ss)$, where $\CC$ is an additive category, $\E$ is an additive bifunctor $\CC^{\op} \times \CC \to \Ab$, and $\ss$ is so called a \emph{realization} of $\E$, which designates the class of \emph{conflations}.
We call a pair $(\E,\ss)$ an \emph{extriangulated structure} on $\CC$.

As for exact categories, the author gave in \cite{eno} a classification of possible exact structures on a given idempotent complete additive category $\CC$. More precisely, for an exact structure $\EE$ on an idempotent complete additive category $\CC$, we can associate the category of \emph{defects} $\ddef\EE$, which is a subcategory of the category $\Mod\CC$ of $\CC$-modules.
Then it was shown in \cite{eno} that this gives a bijection between the set of possible exact structures on $\CC$ and the set of Serre subcategories of the category of finitely generated $\CC$-modules satisfying a kind of $2$-regular conditions.

The aim of this paper is to give an analogue of this result for extriangulated structures by using \emph{defects for an extriangulated category} introduced by Ogawa \cite{ogawa}. Unfortunately, we cannot provide a full classification of possible extriangulated structures. This task seems to be rather difficult, since such a classification would give a classification of possible triangulated structures. As another reason, in general, there may be several different triangulated structures on a given category, but the category of defects over triangulated categories are all the same. Thus we cannot expect a classification using defects.

Instead, we give a classification of all possible \emph{substructures} of a \emph{fixed} extriangulated structure. More precisely, let $\CC$ be an additive category and $(\E,\ss)$ a fixed extriangulated structure on $\CC$. Then a subbifunctor $\F$ of $\E$ is \emph{closed} if the natural restriction $(\F,\ss|_\F)$ also gives an extriangulated structure (see Definition \ref{def:closed}).
The goal of this paper is to give a classification of all possible closed subbifunctors of a given  extriangulated category $(\CC,\E,\ss)$.

As in the case of exact categories, we can define the category $\ddef\E$ of \emph{defects} for an extriangulated category $(\CC,\E,\ss)$ (see Definition \ref{def:eff}).
The key observation is Ogawa's result \cite{ogawa}, which shows that $\ddef\E$ is a Serre subcategory of an abelian category of finitely presented $\CC$-modules if $\CC$ has a weak kernel. In this paper, we modify his result to contain extriangulated categories without weak kernels as follows:
In stead of considering finitely presented $\CC$-modules, we consider \emph{coherent} $\CC$-modules, that is, a finitely presented modules such that every finitely generated submodule is finitely presented. For any additive category $\CC$, the category $\coh\CC$ of coherent $\CC$-modules are well-known to be abelian (Proposition \ref{prop:cohwide}). Then we can prove the following.
\begin{propositioni}[= Proposition \ref{prop:defserre}]\label{prop:A}
  Let $(\CC,\E,\ss)$ be a skeletally small extriangulated category. Then the category of defect $\ddef\E$ is a Serre subcategory of an abelian category $\coh\CC$.
\end{propositioni}
This proposition in particular implies that for a skeletally small exact category $(\CC,\EE)$, the category of defect $\ddef\EE$ is a Serre subcategory of $\coh\CC$, which is a new result in its own.

By using this, we can associate to each closed subbifunctor $\F$ of $\E$ a Serre subcategory $\ddef\F$ of $\coh\CC$, which is also a Serre subcategory of $\ddef\E$. Now we are ready to state the main result of this paper.
\begin{theoremi}\label{thm:main}
  Let $(\CC,\E,\ss)$ be a skeletally small extriangulated category. Then the map $\F \mapsto \ddef\F$ gives an isomorphism of the following posets, where the poset structures are given by inclusion.
  \begin{enumerate}
    \item The poset of closed subbifunctors of $\E$.
    \item The poset of Serre subcategories of $\ddef\E$.
  \end{enumerate}
\end{theoremi}

As a biproduct, we can easily deduce a rather simple classification of exact structures. It was shown by Rump \cite{rump} that any additive category $\CC$ has the maximum exact structure $\EE^{\max}$, that is, every exact structure is contained in $\EE^{\max}$.
\begin{corollaryi}[= Corollary \ref{cor:exact}]\label{cor:C}
  Let $\CC$ be a skeletally small exact category. Then the map $\EE \mapsto \ddef\EE$ gives an isomorphism of the following posets:
  \begin{enumerate}
    \item The poset of exact structures on $\CC$.
    \item The poset of Serre subcategories of an abelian category $\ddef\EE^{\max}$.
  \end{enumerate}
\end{corollaryi}
The poset of exact structures was also studied in \cite{BHLR}. This corollary was recently obtained by \cite[Theorem 2.8]{FG} under the assumption of idempotent completeness, whose proof uses results in \cite{eno}. Our proof uses Theorem \ref{thm:main}, is a rather simple, and does not use any results in \cite{eno}.

\begin{org}
  This paper is organized as follows.
  In Section \ref{sec2}, we recall some basic properties of extriangulated categories, and study the category of defects to show Proposition \ref{prop:A}.
  In Section \ref{sec3}, we give a proof of Theorem \ref{thm:main}.
  In Section \ref{sec4}, we discuss applications to exact structures, and prove Corollary \ref{cor:C}.
\end{org}

\begin{conv}
  Throughout this paper, we assume that \emph{all categories and functors are additive} and \emph{all subcategories are additive, full and closed under isomorphisms}.
\end{conv}

\section{Preliminaries}\label{sec2}
In this section, we give basic properties of extriangulated categories, and then investigate the category of defects.

\subsection{Basic properties of extriangulated categories}
In this paper, we omit the detailed axioms and definitions on extriangulated categories, and only give terminologies and properties which we shall need later. For the precise definition and the detailed properties of extriangulated categories, we refer the reader to \cite{NP,LN}.

Let $\CC$ be an additive category. Consider the following data $(\E,\ss)$.
\begin{itemize}
  \item $\E$ is an additive bifunctor $\E \colon \CC^{\op} \times \CC \to \Ab$, where $\Ab$ denotes the category of abelian groups.
  \item $\ss$ is a correspondence which associate to each $\delta \in \E(C,A)$ an equivalence class of complexes $[A \to B \to C]$ in $\CC$.
\end{itemize}
Here two complexes $[A \xrightarrow{x} B \xrightarrow{y} C]$ and $[A \xrightarrow{x'} B' \xrightarrow{y'} C ]$ are \emph{equivalent} if there is an isomorphism $b \colon B \to B'$ which makes the following diagram commute:
\[
\begin{tikzcd}
  A \dar[equal]\rar["x"] & B \rar["y"] \dar["b"', "\sim"] & C \dar[equal]\\
  A \rar["x'"'] & B' \rar["y'"'] & C
\end{tikzcd}
\]
For such a data $(\E,\ss)$, we call a complex $A \xrightarrow{x} B \xrightarrow{y} C$ an \emph{$\ss$-conflation} if its equivalence class is equal to $\ss(\delta)$ for some $\delta \in \E(C,A)$. In this case, we call $x$ an \emph{$\ss$-inflation} and $y$ an \emph{$\ss$-deflation}, and we say that this complex \emph{realizes} $\delta$. We write this situation as follows:
\[
\begin{tikzcd}
  A \rar["x"] & B \rar["y"] & C \rar[dashed, "\delta"] & \
\end{tikzcd}
\]

Let $\delta \in \E(C,A)$ and $A \xrightarrow{a} A'$ and $C' \xrightarrow{c} C$ two morphisms in $\CC$. Then we have an element $\E(c,a)(\delta) \in \E(C',A')$. We often write $\E(\id_C,a)(\delta) = a_* \delta$ and $\E(c,\id_A) = c^* \delta$. Then the functoriality of $\E$ implies $\E(c,a)(\delta) = a_* (c^* \delta) = c^* (a_* \delta)$.

Let $\delta_i \in \E(C,A)$ for $i=1,2$. Then a \emph{morphism} $\delta_1 \to \delta_2$ consists of a pair of morphisms $(a,c)$ with $ A_1 \xrightarrow{a} A_2$ and $C_1 \xrightarrow{c} C_2$ satisfying $a_* \delta_1 = c^* \delta_2$.
We say that $\ss$ is a \emph{realization of $\E$} if for every morphism $(a,c) \colon \delta_1 \to \delta_2$ with $\delta_i \in \E(C_i,A_i)$, there is a commutative diagram
\[
\begin{tikzcd}
  A_1 \rar["x_1"] \dar["a"']& B_1 \rar["y_1"] \dar["b"]& C_1 \dar["c"] \rar[dashed, "\delta_1"] & \ \\
  A_2 \rar["x_2"'] & B_2 \rar["y_2"'] & C_2 \rar[dashed, "\delta_2"'] & \ \\
\end{tikzcd}
\]
such that the top complex (resp. the bottom complex) is a realization of $\delta_1$ (resp. $\delta_2$). In this case, we call $(a,b,c)$ a \emph{morphism of $\ss$-conflations}, or a \emph{realization of $(a,c) \colon \delta_1 \to \delta_2$}. Whenever we write diagrams like this, we require that $(a,b,c)$ is a realization of a morphism $(a,c) \colon \delta_1 \to \delta_2$.

Now suppose that $\ss$ is a realization of $\EE$. An \emph{extriangulated category} consists of a triple $(\CC,\E,\ss)$ satisfying some axioms, for which we refer the reader to \cite{NP}. In this case, we call $(\E,\ss)$ an \emph{ on $\CC$}.

We recall some properties of extriangulated categories. Let $(\CC,\E,\ss)$ be an extriangulated category.
One of the axiom is used later:
\begin{itemize}
  \item[(ET3)$^{\op}$] Let $A_i \xrightarrow{x_i} B_i \xrightarrow{y_i} C_i$ be a realization of $\delta_i \in \E(C_i,A_i)$ for $i=1,2$, and suppose that we have the following commutative diagram in $\CC$:
  \[
  \begin{tikzcd}
    A_1 \rar["x_1"] \dar[dashed] & B_1 \rar["y_1"] \dar["b"]& C_1 \dar["c"] \rar[dashed, "\delta_1"] & \ \\
    A_2 \rar["x_2"'] & B_2 \rar["y_2"'] & C_2 \rar[dashed, "\delta_2"'] & \ \\
  \end{tikzcd}
  \]
  Then there exists a morphism $(a,c)\colon \delta_1 \to \delta_2$ which is realized by $(a,b,c)$.
\end{itemize}

The following is also a basic property of extriangulated categories.
\begin{proposition}[{\cite[Proposition 3.3]{NP}}]\label{prop:longexact}
  Let $(\CC,\E,\ss)$ be an extriangulated category, and suppose that we have the following morphism of $\ss$-conflations:
  \[
  \begin{tikzcd}
    A_1 \rar["x_1"] \dar["a"'] & B_1 \rar["y_1"] \dar["b"]& C_1 \dar["c"] \rar[dashed, "\delta_1"] & \ \\
    A_2 \rar["x_2"'] & B_2 \rar["y_2"'] & C_2 \rar[dashed, "\delta_2"'] & \ \\
  \end{tikzcd}
  \]
  Then the following diagram is exact and commutative,
  \[
  \begin{tikzcd}
    \CC(-,A_1) \rar["{\CC(-,x_1)}"] \dar["{\CC(-,a)}"'] & \CC(-,B_1) \rar["{\CC(-,y_1)}"] \dar["{\CC(-,b)}"]& \CC(-,C_1) \dar["{\CC(-,c)}"] \rar["(\delta_1)_\sharp"] & \E(-,A_1) \dar["{\E(-,a)}"]\\
    \CC(-,A_2) \rar["{\CC(-,x_2)}"'] & \CC(-,B_2) \rar["{\CC(-,y_2)}"'] & \CC(-,C_2) \rar[dashed, "(\delta_2)_\sharp"'] & \E(-,A_2) \\
  \end{tikzcd}
  \]
  where $(\delta_i)_\sharp$ for $i=1,2$ is a map corresponding to $\delta_i \in \E(C_i,A_i)$ via the Yoneda lemma.
\end{proposition}

For an extriangulated category $(\CC,\E,\ss)$, the notion of \emph{closed subbifunctors of $\E$} was introduced in \cite{HLN,INP}.
\begin{definition}\label{def:closed}
  Let $(\CC,\E,\ss)$ be an extriangulated category. Then an additive subbifunctor $\F$ of $\E$ is called \emph{closed} if $\ss|_\F$-deflations are closed under compositions.
\end{definition}
Actually, the original definition is different from those in \cite{HLN,INP}, but is equivalent to ours by \cite[Proposition 5.5]{INP}. The most important property of closed subbifunctors is the following.
\begin{proposition}[{\cite[Proposition 5.5]{INP}}]\label{prop:closed}
  Let $(\CC,\E,\ss)$ be an extriangulated category and $\F$ an additive subbifunctor of $\E$. Then the following are equivalent.
  \begin{enumerate}
    \item $\F$ is closed.
    \item $(\F,\ss|_\F)$ gives an extriangulated structure of $\CC$.
  \end{enumerate}
\end{proposition}

We state the following properties of extriangulated categories which will be needed later.
\begin{lemma}\label{lem:tec}
  Let $(\CC,\E,\ss)$ be an extriangulated category. Suppose that we have the following morphism of $\ss$-conflations.
  \[
  \begin{tikzcd}
    A' \rar["x'"]\dar["a"'] & B' \rar["y'"]\dar["b"] & C' \dar["c"] \rar[dashed, "\delta'"] & \ \\
    A \rar["x"'] & B \rar["y"'] & C \rar[dashed, "\delta"'] & \
  \end{tikzcd}
  \]
  Then there exist the following morphisms of $\ss$-conflations,
  \[
  \begin{tikzcd}
    A' \ar[rd, phantom]\rar["x'"]\dar["a"'] & B' \rar["y'"]\dar["b_1"] & C' \dar[equal] \rar[dashed,"\delta'"] & \ \\
    A \rar \dar[equal] & D \ar[rd,phantom]\rar \dar["b_2"] & C' \dar["c"] \rar[dashed, "\delta''"] & \ \\
    A \rar["x"'] & B \rar["y"'] & C \rar[dashed, "\delta"'] & \
  \end{tikzcd}
  \]
  where $\delta'' := c^* \delta = a_* \delta'$.
\end{lemma}
\begin{proof}
  By assumption we have $c^* \delta = a_* \delta' \in \E(C',A)$ and put this element $\delta''$. Take a realization $A \to D \to C'$ of $\delta''$. Then we have a morphisms $(c,\id_A) \colon \delta' \to \delta''$ and $(\id_C,a) \colon \delta'' \to \delta$. By taking realizations of these morphisms, we obtain the desired result.
\end{proof}

\begin{lemma}\label{lem:ln}
  Let $(\CC,\E,\ss)$ be an extriangulated category and $A \xrightarrow{x} B \xrightarrow{y} C$ be an $\ss$-conflation which realizes $\delta \in \E(C,A)$. Then for any $C' \xrightarrow{c} C$ in $\CC$, there exists the following morphism of $\ss$-conflations which realizes $(\id_A,c) \colon c^* \delta \to \delta$
  \[
  \begin{tikzcd}
    A \rar["x'"] \dar[equal] & B' \rar["y'"]\dar["b"'] \ar[rd,phantom, "(*)"]& C' \dar["c"] \rar[dashed, "{c^* \delta}"] &\ \\
    A \rar["x"'] & B \rar["y"'] & C \rar[dashed, "\delta"'] & \
  \end{tikzcd}
  \]
  such that $(*)$ is a weak pullback.
\end{lemma}
\begin{proof}
  This immedately follows from \cite[Proposition 1.20]{LN}.
\end{proof}

\subsection{Category of coherent modules and defects}
In this subsection, we first give basic properties of the category of coherent modules over a skeletally small additive category, and then give a modification of Ogawa's result on the category of defects over an extriangulated category.

Let $\CC$ be a skeletally small additive category. Then a \emph{right $\CC$-module} is a contravariant functor $\CC^{\op} \to \Ab$ to the category of abelian groups, and a morphism between right $\CC$-modules is just a natural transformation of functors. By this, all right $\CC$-modules form a Grothendieck abelian category $\Mod\CC$ with enough projectives, and projective objects are direct summands of (possibly infinite) direct sums of representable functors.

Now we introduce several finiteness conditions on $\CC$-modules.
\begin{definition}
  Let $\CC$ be a skeletally small additive category and $M \in \Mod\CC$.
  \begin{enumerate}
    \item $M$ is \emph{finitely generated} if there is a surjection $\CC(-,C) \defl M$ for some $C \in \CC$.
    \item $M$ is \emph{finitely presented} if there is an exact sequence of the following form.
    \[
    \begin{tikzcd}
      \CC(-,B) \rar & \CC(-,C) \rar & M \rar & 0
    \end{tikzcd}
    \]
    \item $M$ is \emph{coherent} if $M$ is finitely presented and every finitely generated submodule of $M$ is finitely presented. We denote by $\coh \CC$ the category of coherent right $\CC$-modules.
  \end{enumerate}
\end{definition}
The following property is well-known. For the proof, see \cite[Proposition 1.5]{herzog} for example.
\begin{proposition}\label{prop:cohwide}
  Let $\CC$ be a skeletally small additive category. Then $\coh\CC$ is a wide subcategory of $\Mod\CC$, that is, $\coh\CC$ is closed under kernels, cokernels and extensions in $\Mod\CC$. In particular, $\coh\CC$ is an exact abelian subcategory of $\Mod\CC$.
\end{proposition}

Next we introduce \emph{defects} of $\ss$-conflations following \cite{ogawa}.
\begin{definition}[{\cite[Definition 2.4]{ogawa}}]\label{def:eff}
  Let $(\CC,\E,\ss)$ be a skeletally small extriangulated category.
  \begin{enumerate}
    \item Let $\delta \in \E(C,A)$. Then take a realization $A \to B \xrightarrow{y} C$ of $\delta$, and define $\widetilde{\delta}$ as the cokernel of $\CC(-,y) \colon \CC(-,B) \to \CC(-,C)$ in $\Mod\CC$.
    We call $\widetilde{\delta}$ a \emph{defect of $\delta$}, or a \emph{defect of an $\ss$-conflation $A \to B \xrightarrow{y} C$}.
    \item We denote by $\ddef\E$ the subcategory of $\Mod\CC$ consisting of $\CC$-modules which are isomorphic to defects of some $\ss$-conflations.
  \end{enumerate}
\end{definition}

Now the following is a key observation in this paper. Recall that a \emph{Serre subcategory} of an abelian category is a subcategory which is closed under taking subobjects, quotients and extensions.
\begin{proposition}\label{prop:defserre}
  Let $(\CC,\E,\ss)$ be a skeletally small extriangulated category. Then $\ddef\E$ is a Serre subcategory of $\coh\E$.
\end{proposition}
If $\E$ has weak kernels, then $\coh\E$ coincides with the category of finitely presented $\CC$-modules. Thus this recovers \cite[Proposition 2.5]{ogawa}.

We prepare some lemmas to prove this proposition.
\emph{In the rest of this subsection this section, we denote by $(\CC,\E,\ss)$ a skeletally small extriangulated category.}

We have the following basic property of $\ddef\E$.
\begin{lemma}\label{lem:defclosed}
  The category $\ddef\E$ is closed under kernels and cokernels in $\Mod\CC$.
\end{lemma}
\begin{proof}
  The same argument as in \cite[Lemma 2.6]{ogawa} applies, so we omit the proof.
\end{proof}

The notion of \emph{effaceable $\CC$-modules} is useful, which has a close relation to defects.
\begin{definition}[{\cite[(2.1)]{ogawa}}]
  Let $M \in \Mod\CC$. Then we say that $M$ is \emph{effaceable with respect to $\E$} if the following condition $(*)$ is satisfied:
  \begin{itemize}
    \item[$(*)$] For every $C \in \CC$ and $m \in M(C)$, there exists an $\ss$-deflation $B \xrightarrow{y} C$ such that $M(y)(m) = 0 \in M(B)$.
  \end{itemize}
  We denote by $\Eff\F$ the subcategory of $\Mod\CC$ consists of effaceable $\CC$-modules with respect to $\E$.
\end{definition}
We have the following basic property of $\Eff\F$.
\begin{lemma}\label{lem:effserre}
  The category $\Eff\F$ is a Serre subcategory of $\Mod\CC$.
\end{lemma}
\begin{proof}
  The same argument as in \cite[Lemma 2.3]{ogawa} applies, so we omit the proof.
\end{proof}

The key lemma for Proposition \ref{prop:defserre} is the following, which describes the relation between defects and effaceable modules.
\begin{lemma}\label{lem:defeff}
  We have $\ddef\F = \Eff\F \cap \coh\CC$.
\end{lemma}
\begin{proof}
  We first show $\ddef\F \subset \Eff\F \cap \coh\CC$. The inclusion $\ddef\F \subset \Eff\F$ follows from the same argument as in the latter half of the proof of \cite[Proposition 2.5]{ogawa}, so we omit it.

  Let us show $\ddef\F \subset \coh\CC$. Let $M \in \ddef\F$. Then there is an $\ss$-conflation $A \xrightarrow{x} B \xrightarrow{y} C$ which realizes $\delta \in \E(C,A)$ and satisfies $\coker\CC(-,y) \iso M$. In particular, $M$ is finitely presented.

  Take a finitely generated submodule $L$ of $M$. Since $L$ is finitely generated, there is a surjection $\CC(-,C') \defl L$. Thus we obtain the following exact commutative diagram in $\Mod\CC$.
  \[
  \begin{tikzcd}
    & & \CC(-,C') \rar\dar[dashed] & L \rar\dar[hookrightarrow] & 0 \\
    \CC(-,A) \rar["{\CC(-,x)}"]  & \CC(-,B) \rar["{\CC(-,y)}"]  & \CC(-,C) \rar & M \rar & 0
  \end{tikzcd}
  \]
  Here the dashed map exists by the projectivity of $\CC(-,C')$. Then by the Yoneda lemma, there is a corresponding map $c \colon C' \to C$ in $\CC$.
  Now consider $c^* \delta \in \CC(C',A)$. By Lemma \ref{lem:ln}, we have the following realization of a morphism $(c,\id_A) \colon c^*\delta \to \delta$,
  \[
  \begin{tikzcd}
    A \rar["x'"] \dar[equal]& B' \ar[rd,phantom, "(*)"]\rar["y'"] \dar["b"']& C' \dar["c"] \rar[dashed, "c^* \delta"] & \ \\
    A \rar["x_2"'] & B \rar["y_2"'] & C \rar[dashed, "\delta"'] & \ \\
  \end{tikzcd}
  \]
  such that $(*)$ is a weak pullback.
  Now by the Yoneda embedding, we obtain the following commutative diagram in $\Mod\CC$.
  \[
  \begin{tikzcd}
    \CC(-,B') \rar["{\CC(-,y')}"] \dar["{\CC(-,b)}"']& \CC(-,C') \rar\dar["{\CC(-,c)}"] & L \rar\dar[hookrightarrow] & 0 \\
    \CC(-,B) \rar["{\CC(-,y)}"']  & \CC(-,C) \rar & M \rar & 0
  \end{tikzcd}
  \]
  Using the fact that $(*)$ is a weak pullback, we can check that the top sequence above is exact. Therefore, $L$ belongs to $\ddef\CC$.

  Next we show $\Eff\F \cap \coh\CC \subset \ddef\F$. This part is a modification of the first half of the proof of \cite[Proposition 2.5]{ogawa}.
  Take a coherent $\CC$-module $M$ which is effaceable with respect to $\F$. In particular, since $M$ is finitely generated, there is a surjection $\CC(-,C) \defl M$. By the Yoneda lemma, we can identify this map with an element $m \in M(C)$. Since $M$ is effaceable, there is an $\ss$-deflation $B \xrightarrow{y} C$ satisfying $M(y)(m)=0$. Consider the following exact commutative diagram.
  \[
  \begin{tikzcd}
    \CC(-,B) \rar["{\CC(-,y)}"] \ar[rd,"0"'] & \CC(-,C) \rar \dar[twoheadrightarrow, "m"] & \widetilde{\delta} \ar[dl,twoheadrightarrow, dashed] \rar & 0 \\
    & M
  \end{tikzcd}
  \]
  Here the composition $m \circ \CC(-,y)$ is zero by $M(y)(m)=0$. Therefore there is a dashed map which makes the above diagram commute. Now we obtain the following exact sequence in $\Mod\CC$:
  \[
  \begin{tikzcd}
    0 \rar & K \rar & \widetilde{\delta} \rar & M \rar & 0
  \end{tikzcd}
  \]
  Since $M$ is finitely presented and $\widetilde{\delta}$ is finitely generated, we have that $K$ is finitely generated (see e.g. \cite{herzog}). Moreover, since $\ddef\E \subset \Eff\F$ and $\Eff\F$ is closed under subobjects in $\Mod\CC$ by Lemma \ref{lem:effserre}, we have $K \in \Eff\F$.
  Now the same argument shows that we have a surjection $\widetilde{\delta'}\defl K$ with $\widetilde{\delta'} \in \ddef\CC$.
  Thus $M$ is a cokernel of a morphism $\widetilde{\delta'} \to \widetilde{\delta}$, thus belongs to $\ddef\CC$ since $\ddef\CC$ is closed under cokernels in $\Mod\CC$ by Lemma \ref{lem:defclosed}.
\end{proof}

Now Lemmas \ref{lem:effserre} and \ref{lem:defeff} immediately imply Proposition \ref{prop:defserre}.
\begin{proof}[Proof of Proposition \ref{prop:defserre}]
  We have $\ddef\CC = \coh\CC \cap \Eff\CC$ by Lemma \ref{lem:defeff}, and $\Eff\CC$ is a Serre subcategory of $\Mod\CC$ by Lemma \ref{lem:effserre}. These facts immediately imply that $\ddef\CC$ is a Serre subcategory of $\coh\CC$.
\end{proof}

\section{Proof of the main theorem}\label{sec3}
In this section, we will give a proof of Theorem \ref{thm:main}.
\emph{Throughout this section, we denote by $(\CC,\E,\ss)$ a skeletally small extriangulated category.}
In this setting, $\ddef\E$ is an abelian category since it is a Serre subcategory of an abelian category $\coh\CC$ by Proposition \ref{prop:defserre}.

We will construct two maps between (1) and (2) in Theorem \ref{thm:main}, that is:
\begin{enumerate}
  \item The set of closed subbifunctors of $\E$.
  \item The set of Serre subcategories of $\ddef\E$.
\end{enumerate}

{\bf A map $\ddef(-)$ from (1) to (2)}.

Let $\F$ be a closed subbifunctor of $\E$. Then $(\F,\ss|_\F)$ is an extriangulated structure by Proposition \ref{prop:closed}. Thus $\ddef\F$ is a Serre subcategory of $\coh\CC$ by Proposition \ref{prop:defserre}.
Since every $\F$-deflation is an $\E$-deflation, $\ddef\F$ is a subcategory of $\ddef\E$. Thus clearly $\ddef\F$ is a Serre subcategory of an abelian category $\ddef\E$, thus we obtain a map $\ddef(-)$ from (1) to (2).

{\bf A map $\F(-)$ from (2) to (1)}.

Let $\SS$ be a Serre subcategory  of $\ddef\E$, and we will construct an additive subbifunctor $\F(\SS)$ of $\E$. Since $\ddef\E$ is a Serre subcategory of $\coh\CC$ by Proposition \ref{prop:defserre}, it immediately follows that $\SS$ is also a Serre subcategory of $\coh\CC$.

Let $A$ and $C$ be objects in $\CC$ and $\delta \in \E(C,A)$. Then we denote by $\F(\SS)(C,A)$ a subset of $\E(C,A)$ consisting of $\delta$ such that its defect $\widetilde{\delta}$ belongs to $\SS$.

We claim that $\F(\SS)$ actually defines a closed subbifunctor of $\E$. We divide its proof into 3 steps.

\un{{\bf (Step 1)} $\F$ is a subbifunctor of $\E$.} Let $\delta \in \F(\SS)(C,A)$ and $\ss(\delta) = [A \xrightarrow{x} B \xrightarrow{y} C]$. First take any $c \colon C' \to C$, then we show that $c^* \delta$ belongs to $\F(\SS)(C',A)$, that is, $\widetilde{c^* \delta} \in \SS$.
By Lemma \ref{lem:ln}, we have the following morphism of $\ss$-conflations
\[
\begin{tikzcd}
  A \rar["x'"] \dar[equal] & B' \rar["y'"]\dar["b"'] \ar[rd,phantom, "(*)"]& C' \dar["c"] \rar[dashed, "{c^* \delta}"] &\ \\
  A \rar["x"'] & B \rar["y"'] & C \rar[dashed, "\delta"'] & \
\end{tikzcd}
\]
such that $(*)$ is a weak pullback. Now by sending this diagram to $\Mod\CC$ under the Yoneda embedding, we obtain the following exact commutative diagram, where $L$ and $M$ are defects of $c^* \delta$ and $\delta$ respectively.
\[
\begin{tikzcd}
  \CC(-,A) \rar["{\CC(-,x')}"] \dar[equal] & \CC(-,B') \rar["{\CC(-,y')}"] \dar["{\CC(-,b)}"']  & \CC(-,C')\dar["{\CC(-,c)}"] \dar \rar & L \dar["\iota"] \rar & 0 \\
  \CC(-,A) \rar["{\CC(-,x)}"']  & \CC(-,B) \rar["{\CC(-,y)}"']  & \CC(-,C) \rar & M \rar & 0
\end{tikzcd}
\]
Since $\delta$ belongs to $\F(\SS)$, we have $M \in \SS$. On the other hand, since $(*)$ is a weak pullback, it is easily checked that $\iota$ is an injection. Now we have a short exact sequence
\[
\begin{tikzcd}
  0 \rar & L \rar["\iota"] & M \rar & \coker \iota \rar & 0
\end{tikzcd}
\]
in $\Mod\CC$. Recall that $L,M \in \ddef\E \subset \coh\CC$ by Proposition \ref{prop:defserre}.
Since $\coh\CC$ is closed under cokernels in $\Mod\CC$ by Proposition \ref{prop:cohwide}, we have $\coker \iota \in \coh\CC$. Therefore, since $\SS$ is a Serre subcategory of $\coh\CC$ and $M \in \SS$, we have $L \in \SS$.
This means $c^* \delta \in \F(\SS)(C',A)$.

Next take any $a \colon A \to A'$, then we show that $a_* \delta$ belongs to $\F(\SS)(C,A')$.
Then we have the following morphism of $\ss$-conflations which realizes $(a,\id_C) \colon \delta \to a_* \delta$.
\[
\begin{tikzcd}
  A \rar["x"]\dar["a"'] & B \rar["y"] \dar["b"] & C\dar[equal] \rar[dashed, "{\delta}"] & \ \\
  A' \rar["x'"'] & B' \rar["y'"']& C \rar[dashed, "a_* \delta"'] &\
\end{tikzcd}
\]
Now by sending this diagram to $\Mod\CC$ under the Yoneda embedding, we obtain the following exact commutative diagram, where $M$ and $N$ are defects of $\delta$ and $a_*\delta$ respectively.
\[
\begin{tikzcd}[column sep = large]
  \CC(-,A) \rar["{\CC(-,x)}"] \dar["{\CC(-,a)}"'] & \CC(-,B) \rar["{\CC(-,y)}"] \dar["{\CC(-,b)}"]  & \CC(-,C) \dar[equal] \rar & M \dar["\pi"] \rar & 0 \\
  \CC(-,A') \rar["{\CC(-,x')}"']  & \CC(-,B') \rar["{\CC(-,y')}"']  & \CC(-,C) \rar & N \rar & 0
\end{tikzcd}
\]
Then $\pi$ is a surjection. Now we have a short exact sequence
\[
\begin{tikzcd}
  0 \rar & \ker \pi \rar & M \rar["\pi"] & N \rar & 0
\end{tikzcd}
\]
in $\Mod\CC$. Recall that $M,N \in \ddef\E \subset \coh\CC$ by Proposition \ref{prop:defserre}.
Since $\coh\CC$ is closed under kernels in $\Mod\CC$ by Proposition \ref{prop:cohwide}, we have $\ker \pi \in \coh\CC$. Therefore, since $\SS$ is a Serre subcategory of $\coh\CC$ and $M \in \SS$, we have $N \in \SS$.
This means $a_* \delta \in \F(\SS)(C,A')$.

Now we have shown that $\F(\SS)$ is a subbifunctor of $\E$. $\qedb$

\un{{\bf (Step 2)} $\F(\SS)$ is an additive subbifunctor of $\E$.} It suffices to show that $\F(\SS)(C,A)$ is an additive subgroup of $\E(C,A)$ for each $A,C \in \CC$. Clearly $0 \in \F(\SS)(C,A)$.

Let $\delta_1,\delta_2 \in \F(\SS)(C,A)$. Then we have $\delta_1 -\delta_2 = \begin{sbmatrix} \id_C \\ -\id_C \end{sbmatrix}_* [\id_A \id_A]^* (\delta_1 \oplus \delta_2)$.
Thus by (Step 1), it suffices to show $\delta\oplus \rho \in \F(\SS)(C \oplus C,A \oplus A)$ to prove $\delta_1 - \delta_2 \in \F(\SS)(C,A)$.

Let $\ss(\delta_i) = [A \xrightarrow{x_i} B_i \xrightarrow{y_i} C]$ for $i=1,2$. Then we have
\[
\ss(\delta_1 \oplus \delta_2) = [ A\oplus A \xrightarrow{x_1 \oplus x_2} B_1 \oplus B_2 \xrightarrow{y_1 \oplus y_2} C \oplus C ]
\]
by the axiom of an extriangulated category, see \cite[Definition 2.10]{NP}. Therefore, $\widetilde{\delta_1 \oplus \delta_2} = \widetilde{\delta_1} \oplus \widetilde{\delta_2}$ holds, and we have  $\widetilde{\delta_1},\widetilde{\delta_2} \in \SS$ by $\delta_1,\delta_2 \in \F(\SS)(C,A)$.
Therefore, the defect of $\delta_1\oplus\delta_2$ belongs to $\SS$.
This means $\delta_1 \oplus \delta_2 \in \F(\SS)(C \oplus C, A \oplus A)$. $\qedb$

\un{{\bf (Step 3)} $\F(\SS)$ is a closed subbifunctor of $\E$.}
We will show that $\ss|_{\F(\SS)}$-deflations are closed under compositions.
Let $X \xrightarrow{x} Y$ and $Y \xrightarrow{y} Z$ be two $\ss|_{\F(\SS)}$-deflations. Then since these are $\ss$-deflations, their composition $X \xrightarrow{yx} Z$ is an $\ss$-deflation by the axiom (ET4) of extriangulated categories, see \cite[Definition 2.12]{NP}.
Now we obtain the following commutative diagram in $\Mod\CC$, and $L ,N \in \SS$ by the definition of $\F(\SS)$.
\[
\begin{tikzcd}
  \CC(-,X) \dar[equal] \rar["{\CC(-,x)}"] & \CC(-,Y) \dar["{\CC(-,y)}"]\rar & L\dar["f"] \rar & 0 \\
  \CC(-,X) \dar["{\CC(-,x)}"'] \rar["{\CC(-,yx)}"] & \CC(-,Z) \dar[equal]\rar & M \dar["g"] \rar & 0\\
  \CC(-,Y) \rar["{\CC(-,y)}"'] & \CC(-,Z) \rar & N \rar & 0
\end{tikzcd}
\]
Now the diagram chasing shows that the following is exact in $\Mod\CC$:
\[
\begin{tikzcd}
  0 \rar & \ker f \rar & L \rar["f"] & M \rar["g"] & N \rar & 0
\end{tikzcd}
\]
Now consider the following two exact sequene in $\Mod\CC$:
\[
\begin{tikzcd}[row sep = tiny]
  0 \rar & \ker f \rar & L \rar & \im f \rar & 0, \\
  0 \rar & \im f \rar & M \rar & N \rar & 0.
\end{tikzcd}
\]
We have $L,M,N \in \ddef\E \subset \coh\CC$. Thus all the terms belongs to $\coh\CC$ by Proposition \ref{prop:cohwide}.
Since $\SS$ is a Serre subcategory of $\coh\CC$ and $L,N \in \SS$, the first exact sequence implies $\im f \in \SS$, and the second implies $M\in\SS$. This means that $yx$ is an $\ss|_{\F(\SS)}$-deflation. $\qedb$

Therefore, we obtain a map $\F(-)$ from (2) to (1). It is obvious that maps $\F(-)$ and $\ddef(-)$ are order-preserving. We are going to prove that these maps are mutually inverse to each other.

\un{$\ddef \F(\SS) = \SS$ holds for a Serre subcategory $\SS$ of $\ddef\E$.}

Let $M \in \ddef \F(\SS)$. Then $M$ is isomorphic to $\widetilde{\delta}$ for some $\delta \in \F(\SS)(C,A)$ by the definition of $\ddef(-)$. On the other hand, by the definition of $\F(\SS)$, we have $\widetilde{\delta} \in \SS$. Thus $M$ belongs to $\SS$.

Conversely, let $M \in \SS$. Then since $\SS \subset \ddef\E$, we have that $M$ is isomorphic to $\widetilde{\delta}$ for some $\delta \in \E(C,A)$. This implies that $\widetilde{\delta}$ belongs to $\SS$, thus $\delta \in \F(\SS)(C,A)$. Therefore, we have $M \iso \widetilde{\delta}\in \ddef \F(\SS)$.

\un{$\F(\ddef \F) = \F$ holds for a closed bifunctor $\F$ of $\E$.}

By definition, $\F(C,A) \subset \F(\ddef\F)(C,A)$ can be easily checked.
The converse direction is the most complicated part of the proof. The key observation is the following lemma.
\begin{lemma}\label{lem:key}
  Let $(\CC,\E,\ss)$ be an extriangulated category and $\F$ a closed subbifunctor of $\E$. Suppose that we have two elements $\delta \in \E(C,A)$ and $\delta' \in \F(C',A')$. If $\widetilde{\delta}$ and $\widetilde{\delta'}$ are isomorphic, then $\delta \in \F(C,A)$ holds.
\end{lemma}
\begin{proof}
  Let $\ss(\delta) =[A \xrightarrow{x} B \xrightarrow{y} C]$ and $\ss(\delta') = [A' \xrightarrow{x'} B' \xrightarrow{y'} C']$. By assumption, we have the following commutative exact diagram in $\Mod\CC$.
  \[
  \begin{tikzcd}
    \CC(-,A') \rar["{\CC(-,x')}"] & \CC(-,B')\dar[dashed] \rar["{\CC(-,y')}"] & \CC(-,C') \rar\dar[dashed] & M \dar[equal] \rar & 0 \\
    \CC(-,A) \rar["{\CC(-,x)}"'] & \CC(-,B) \rar["{\CC(-,y)}"'] & \CC(-,C) \rar & M  \rar & 0 \\
  \end{tikzcd}
  \]
  Then there are dashed arrows which make the diagram commute by the projectivity of representable functors. By the Yoneda lemma, we obtain the following commutative diagram in $\CC$:
  \[
  \begin{tikzcd}
    A' \rar["x'"]\dar[dashed, "a"'] & B' \rar["y'"]\dar["b"] & C' \dar["c"] \rar[dashed, "\delta'"] & \ \\
    A \rar["x"'] & B \rar["y"'] & C \rar[dashed, "\delta"'] & \
  \end{tikzcd}
  \]
  Now by (ET3)$^{\op}$, there exists a dashed map $a$ such that the above diagram commutes and $c^* \delta = a_* \delta'$. Then a morphism $(c,a) \colon\delta' \to \delta$ is realized by $(a,b,c)$. Thus by Lemma \ref{lem:tec}, we obtain the following morphisms of $\ss$-conflations.
  \[
  \begin{tikzcd}
    A' \rar["x'"]\dar["a"'] & B' \rar["y'"]\dar["b_1"] & C' \dar[equal] \rar[dashed, "\delta'"] & \ \\
    A \rar \dar[equal] & D \rar \dar["b_2"] & C' \dar["c"] \rar[dashed, "a_* \delta'"] & \ \\
    A \rar["x"'] & B \rar["y"'] & C \rar[dashed, "\delta"'] & \
  \end{tikzcd}
  \]
  Again by the Yoneda embedding, we obtain the following exact commutative diagram.
  \[
  \begin{tikzcd}
    \CC(-,A') \rar["{\CC(-,x')}"]\dar["{\CC(-,a)}"'] & \CC(-,B') \dar["{\CC(-,b_1)}"]\rar["{\CC(-,y')}"] & \CC(-,C') \rar\dar[equal] & M \dar["f"'] \ar[dd,bend left = 40, equal] \rar & 0 \\
    \CC(-,A) \dar[equal]\rar & \CC(-,D) \dar["{\CC(-,b_2)}"]\rar & \CC(-,C') \rar\dar["{\CC(-,c)}"] & N \dar["g"']  \rar & 0 \\
    \CC(-,A) \rar["{\CC(-,x)}"'] & \CC(-,B) \rar["{\CC(-,y)}"'] & \CC(-,C) \rar & M  \rar & 0 \\
  \end{tikzcd}
  \]
  Here $gf = \id_M$ holds by the surjectivity of $\CC(-,C') \to M$, thus $f$ is injective.
  On the other hand, $f$ is surjective by the above commutative diagram, hence $f$ is an isomorphism. Thus so is $g$. Therefore we may assume that $g = \id_M$, and obtain the following exact commutative diagram by Proposition \ref{prop:longexact}.
  \[
  \begin{tikzcd}
    \CC(-,A) \ar[dd,equal]\rar & \CC(-,D) \ar[dd,"{\CC(-,b_2)}"']\rar & \CC(-,C') \ar[rr, "{(a_* \delta')_\sharp}"]\ar[dd, "{\CC(-,c)}"'] \ar[rd,twoheadrightarrow] & & \E(-,A) \ar[dd,equal] \\
    & & & M \ar[ru,hookrightarrow] \ar[rd,hookrightarrow]\\
    \CC(-,A) \rar["{\CC(-,x)}"'] & \CC(-,B) \rar["{\CC(-,y)}"'] & \CC(-,C) \ar[ru,twoheadrightarrow]\ar[rr,"{\delta_\sharp}"'] & & \E(-,A) \\
  \end{tikzcd}
  \]
  Recall that $\delta'$ belongs to $\F(C',A')$. Thus $a_* \delta'$ belongs to $\F(C',A)$ since $\F$ is a subbifunctor of $\E$. Then we claim that the image of the map $(a_* \delta')_\sharp \colon \CC(-,C') \to \E(-,A)$, which is isomorphic to $M$, is actually contained in $\F(-,A)$.
  Indeed, let $X \in \CC$ and evaluate $(a_* \delta')_\sharp$ at $X$. Then we obtain a map $\CC(X,C') \to \E(X,A)$ which sends $\varphi$ to $\varphi^* (a_* \delta')$. Now $a_* \delta' \in \F(C',A)$ implies $\varphi^* (a_* \delta') \in \F(X,A)$.

  Therefore, we obtain the following commutative diagram.
  \[
  \begin{tikzcd}
     \CC(-,C') \ar[rr, "{(a_* \delta')_\sharp}"]\ar[dd, "{\CC(-,c)}"'] \ar[rd,twoheadrightarrow] & & \F(-,A) \ar[dd,hookrightarrow] \\
     & M \ar[ru,hookrightarrow] \ar[rd,hookrightarrow]\\
     \CC(-,C) \ar[ru,twoheadrightarrow]\ar[rr,"{\delta_\sharp}"'] & & \E(-,A) \\
  \end{tikzcd}
  \]
  This implies that the image of $\delta_\sharp \colon \CC(-,C) \to \E(-,A)$ is contained in $\F(-,A)$. Now by evaluating this map at $C$, we obtain a map $(\delta_\sharp)_C \colon \CC(C,C) \to \E(C,A)$, whose image is contained in $\F(C,A)$. Therefore, $\delta = (\delta_\sharp)_C (\id_C)$ belongs to $\F(C,A)$.
\end{proof}

Now let us return to the proof of our main theorem.
Take $\delta \in \F(\ddef\F)(C,A)$. Then $\widetilde{\delta}$ belongs to $\ddef\F$ by the definition of $\F(\ddef\F)(C,A)$. Therefore, there exists some $\delta' \in \F(C',A')$ such that $\widetilde{\delta'}$ is isomorphic to $\widetilde{\delta}$.
Then Lemma \ref{lem:key} implies that $\delta \in \F(C,A)$. Therefore, we obtain $\F(\SS)(C,A) \subset \F(C,A)$.
This finishes the proof of Theorem \ref{thm:main}.

\section{Classifications of exact structures revisited}\label{sec4}
In this section, we apply Theorem \ref{thm:main} to study exact structures.
First let us recall basics on exact categories and their relation to extriangulated categories.
Let $\CC$ be an additive category. An \emph{exact category} consists of a pair $(\CC,\EE)$, where $\CC$ is an additive category and $\EE$ is a class of kernel-cokernel pairs in $\CC$ satisfying some axioms. We call $\EE$ an \emph{exact structure} on $\CC$.
For the precise definition and basic properties of exact categories, we refer the reader to \cite{buhler}.

A skeletally small exact category $(\CC,\EE)$ can be regarded as an extriangulated category in the following way:
For $A,C \in \CC$, put $\E_\EE(C,A):= \Ext_\EE^1(C,A)$, where $\Ext_\EE^1(C,A)$ is the set of equivalence classes of complexes in $\EE$ of the form $A \to B \to C$. Then $\E_\EE$ becomes an additive bifunctor $\E_\CC(C,A) \colon \CC^{\op} \times \CC \to \Ab$.
Define the realization $\ss_\EE(\delta)$ of $\delta \in \E_\EE$ to be $\delta$ itself. Then this data $(\E_\EE,\ss_\EE)$ defines an extriangulated structure on $\CC$, see \cite[Example 2.13]{NP} for the detail. Clearly, this operation gives an injection from the set of exact structures on $\CC$ to the set of extriangulated structures on $\CC$.
We have the following characterization of its image.
\begin{proposition}[{\cite[Corollary 3.18]{NP}}]\label{prop:exchar}
  Let $(\CC,\E,\ss)$ be an extriangulated category. Then the following are equivalent.
  \begin{enumerate}
    \item There is an exact structure $\EE$ on $\CC$ satisfying $(\E,\ss) = (\E_\EE,\ss_\EE)$.
    \item Every $\ss$-inflation is a monomorphism, and every $\ss$-deflation is an epimorphism.
  \end{enumerate}
\end{proposition}

We have the following compatibility on closed subbifunctors and exact structures.
\begin{lemma}\label{lem:compati}
  Let $(\CC,\EE)$ be a skeletally small exact category. Then the map $\FF \mapsto \E_\FF$ gives an isomorphism of posets between the following two posets, where poset structures are given by inclusion.
  \begin{enumerate}
    \item The poset of exact structures which is contained in $\EE$.
    \item The poset of closed subbifunctors of $\E_\EE$.
  \end{enumerate}
\end{lemma}
\begin{proof}
  Denote by $(\E_\EE,\ss)$ be an extriangulated structure on $\CC$ corresponding to $\EE$.
  Let $\FF$ be an exact structure on $\CC$ with $\FF \subset \EE$, and  Then clearly we have $\E_\FF = \Ext_\FF^1(-,-)$ is a closed subbifunctor of $\E_\EE = \Ext_\EE^1(-,-)$, thus we obtain a map $\FF \mapsto \E_\FF$ from (1) to (2). Clearly this map is order-preserving, and induces an embedding of posets (1) into (2).
  Therefore, it suffices to show the following claim:

  {\bf (Claim)}: Every closed subbifunctor of $\E_\EE$ is of the form $\E_\FF$ for some exact structure $\FF$ on $\CC$ satisfying $\FF \subset \EE$.

  Let $\F$ be a closed subbifunctor of $\E_\EE$, and consider the extriangulated structure $(\F,\ss|_{\F})$ on $\CC$. Since every $\ss|_{\F}$-inflation (resp. $\ss|_{\F}$-deflation) is an $\ss$-inflation (resp. an $\ss$-deflation), it is a monomorphism (resp. an epimorphism).
  Therefore, Proposition \ref{prop:exchar} implies that there is an exact structure $\FF$ on $\CC$ satisfying $\F = \E_\FF$ and $\ss|_{\F} = \ss_\FF$.
  In addition, since we have $\Ext_\FF^1 = \F \subset \E = \Ext_\EE^1$ and $\ss_\FF$ coincides with the restriction of $\ss$, we obtain $\FF \subset \EE$.
\end{proof}

By abuse of notation, for an exact structure $\EE$ on a skeletally small additive category $\CC$, we put $\ddef\EE := \ddef \E_{\EE}$, that is, $\ddef\EE$ is a subcategory of $\Mod\CC$ consisting of $M$ such that there is an $\EE$-deflation $B \xrightarrow{y} C$ wich $M \iso \coker \CC(-,y)$. This category was used in \cite{eno} to classify all possible exact structures.
It is a Serre subcategory of an abelian category $\coh\EE$ by Proposition \ref{prop:defserre}, thus is an abelian category.

As we mentioned in the introduction, every additive category $\CC$ admits a unique maximal exact structure $\EE^{\max}$ which contains all exact structures on $\CC$ \cite{rump}.
Then we immediately obtain the following corollary.
\begin{corollary}\label{cor:exact}
  Let $\CC$ be a skeletally small additive category. Then there is an isomorphism of the following posets:
  \begin{enumerate}
    \item The poset of exact structures on $\CC$.
    \item The poset of Serre subcategories of an abelian category $\ddef\EE^{\max}$.
  \end{enumerate}
  The map from {\upshape (1)} to {\upshape (2)} is given by $\FF\mapsto \ddef\FF$.
\end{corollary}
\begin{proof}
  Since every exact structure is contained in $\EE^{\max}$, a map $\FF \mapsto \E_\FF$ gives an isomorphism of posets between (1) and the poset of closed subbifunctors of $\E_{\EE^{\max}}$ by Lemma \ref{lem:compati}. Then Theorem \ref{thm:main} shows that the latter poset is isomorphic to (2).
\end{proof}
\begin{remark}
  Note that in \cite{eno}, we assume that $\CC$ is an idempotent complete so that $\Mod\CC$ behaves nicer, but in this paper we do not need this assumption. Actually, Corollary \ref{cor:exact} cannot deduce classification results in \cite{eno} since we do not know an explicit structure of $\EE^{\max}$ or $\ddef\EE^{\max}$. However, our description is a simpler at least in the theoretical sense, and its proof quickly follows from our classification of subbifunctors, with the aid of existence of $\EE^{\max}$ and the well-developed theory of extriangulated categories.
\end{remark}

\begin{remark}
  Recently, Corollary \ref{cor:exact} was shown by Fan-Gorsky \cite[Theorem 2.8]{FG} for the idempotent complete case by using results in \cite{eno}.
  Moreover, they discuss possible extriangulated structures on an additive category in \cite[Section 7]{FG} as an announcement of a forthcoming paper. Actually, the preparation of this paper is partly motivated by their discussion, but the author obtained the main result Theorem \ref{thm:main} independently before their paper.
\end{remark}

\begin{ack}
  The author would like to thank Hiroyuki Nakaoka for helpful discussions and his intensive lecture on extriangulated categories at Nagoya University, which strongly motivates this paper. This work is supported by JSPS KAKENHI Grant Number JP18J21556.
\end{ack}


\begin{thebibliography}{199}

  \bibitem[BHLR]{BHLR}
  T. Br\"ustle, S. Hassoun, D. Langford, S. Roy,
  \emph{Reduction of exact structures},
  J. Pure Appl. Algebra 224 (2020), no. 4, 106212, 29 pp.

  \bibitem[B\"uh]{buhler}
  T. B\"uhler, \emph{Exact categories}, Expo. Math. 28 (2010), no. 1, 1--69.

  \bibitem[Eno]{eno}
  H. Enomoto, \emph{Classifications of exact structures and Cohen-Macaulay-finite algebras}, Adv. Math. 335 (2018), 838--877.

  \bibitem[FG]{FG}
  X. Fang, M. Gorsky, \emph{Exact structures and degeneration of Hall algebras}, arXiv:2005.12130.

  \bibitem[HLN]{HLN}
  M. Herschend, Y. Liu, H. Nakaoka,
  \emph{$n$-exangulated categories}, arXiv:1709.06689.

  \bibitem[Her]{herzog}
  I. Herzog, \emph{The Ziegler spectrum of a locally coherent Grothendieck category}, Proc. London Math. Soc. 74 (1997), 503--558.

  \bibitem[INP]{INP}
  O. Iyama, H. Nakaoka, Y. Palu,
  \emph{Auslander--Reiten theory in extriangulated categories}, arXiv:1805.03776.

  \bibitem[LN]{LN}
  Y. Liu, H. Nakaoka, \emph{Hearts of twin cotorsion pairs on extriangulated categories},
  J. Algebra 528 (2019) 96--149.

  \bibitem[NP]{NP}
  H. Nakaoka and Y. Palu,
  \emph{Extriangulated categories, Hovey twin cotorsion pairs and model structures},
  Cah. Topol. G\'eom. Diff\'er. Cat\'eg. 60(2): 117--193, 2019.

  \bibitem[Oga]{ogawa}
  Y. Ogawa, \emph{Auslander's defects over extriangulated categories: an application for the General Heart Construction}, arXiv:1911.00259.

  \bibitem[Rum]{rump}
  W. Rump, \emph{On the maximal exact structure on an additive category}, Fund. Math. 214 (2011), no. 1, 77--87.

\end{thebibliography}
\end{document}